\documentclass[a4paper,oneside]{article}
\usepackage{amsfonts}
\usepackage{amssymb}
\usepackage{indentfirst}
\usepackage{mathrsfs}
\usepackage[pdftex, pdfpagelabels, bookmarks, hyperindex, hyperfigures,
colorlinks, pagebackref]{hyperref}

\usepackage{cite}
\usepackage{amsfonts,amsmath,latexsym,bm}
\usepackage{color}
\usepackage{amsthm}
\usepackage{float}
\usepackage{amsmath}
\numberwithin{equation}{section}
\newcommand{\tabcaption}{\def\@captype{table}\caption}
\usepackage{graphicx}
\usepackage{graphicx}
\usepackage{threeparttable}
\newtheorem{lem}{Lemma}[section]
\newtheorem{thm}{Theorem}[section]

\newtheorem{rem}{Remark}[section]

\begin{document}

\title{\large\bf An Optimal Embedded Discontinuous Galerkin Method for   Second-Order Elliptic Problems
\thanks
  {
    This work was supported in part by National Natural Science Foundation
    of China (11771312, 11401407) and Major Research Plan of National
    Natural Science Foundation of China (91430105).
  }
  }
\author{
	Xiao Zhang\thanks{Email: zhangxiaofem@163.com}, \ Xiaoping Xie \thanks{Email: xpxie@scu.edu.cn},\
	\ Shiquan Zhang \thanks{Corresponding author. Email: shiquanzhang@scu.edu.cn}\\
	{School of Mathematics, Sichuan University, Chengdu
		610064, China}
	}

\date{}
\maketitle
\begin{abstract}
The  embedded discontinuous Galerkin (EDG) method by Cockburn et al. [SIAM J. Numer. Anal., 2009, 47(4), 2686-2707]
 is obtained from the hybridizable discontinuous Galerkin method by changing the space of the Lagrangian multiplier  from discontinuous  
functions to  continuous ones, and adopts piecewise polynomials of equal degrees   on simplex meshes for all   variables. 
In this paper, we  analyze a new  EDG method for   second order elliptic problems  on  polygonal/polyhedral meshes.  By using piecewise polynomials of degrees $k+1$, $k+1$, $k$ ($k\geq 0$) to approximate the potential, numerical trace and flux, respectively,  the new method is shown to yield optimal convergence rates for both the potential and flux approximations. Numerical experiments are provided to confirm the theoretical results.
\end{abstract}

\section{Introduction}
Let  $\Omega \subset \mathbb{R}^d$ ($d=2,3$)   be a polyhedral domain with boundary $\partial \Omega$. We consider  the following second-order elliptic problem: Find the potential $u$ and the flux $\bm\sigma$ such that
\begin{equation} \label{elliptic}
\left\{\begin{array}{l}
  ~~~~\bm c \bm\sigma -\nabla u =0,~~~~~\text{in}~~\Omega\\
  ~~~~~-div \bm\sigma = f, ~~~~~~\text{in}~~ \Omega\\
  ~~~~~~~~~~~~~u= g, ~~~~~~\text{on}~~ \partial \Omega \\
  \end{array}\right.
\end{equation}
where the diffusion-dispersion tensor $\bm c\in[L^2(\Omega)]^{d\times d}$ is a matrix valued function that is symmetric and uniformly positive definite on $\Omega$, $f \in L^2(\Omega )$, and  { $g\in L^2 (\partial \Omega)$}.\\

In \cite{BJRUH}, Cockburn et al. first proposed a unifying framework for hybridization of finite element methods for second-order elliptic problems. The unifying framework includes as particular cases hybridized versions of mixed methods  \cite{FJJRLD,DNA,BJ}, the continuous Galerkin (CG) method \cite{BCJGHW}, and a wide class of hybridizable discontinuous Galerkin (HDG) methods.
In the HDG framework, the constraint of function continuity on the inter-element boundaries is
relaxed by introducing  numerical traces (Lagrange multipliers)  defined on the inter-element boundaries, thus allowing for piecewise-independent
approximation to the potential or flux solution. By local elimination of the unknowns defined in the interior of elements, the HDG
methods finally lead to symmetric and positive definite (SPD) systems where the unknowns are only the globally coupled degrees of
freedom describing the numerical traces.  We refer to \cite{BJFJPB, BCWQKS, LXAHDG} for some  relevant analyses for the HDG methods.

  The EDG methods  were first proposed in  \cite{SBHKLSP} for   linear shell problems, and then were further studied in \cite{BJSCHK} for second-order elliptic problems. The methods are obtained from the   HDG methods by simply reducing the space of the numerical traces, from piecewise independent to continuous on the whole inter-element boundaries. Since the only degrees of freedom that are globally coupled are precisely those of the numerical traces, such reduction leads to  smaller computational cost of an EDG method   than that of the corresponding HDG method.  Recently, the EDG methods  have been extended to solving several types of fluid flow problems 
  \cite{FNR,NPC,PNC}. 

However, as shown in \cite{BJSCHK}, the EDG methods using  piecewise polynomials of  degree $k(k\geq 1)$    to approximate  all kinds of variables     results in loss of convergence rate for the approximation of flux.
On the other hand, so far  all the EDG methods \cite{BJSCHK,FNR,SBHKLSP,NPC,PNC} are based on simplex meshes, and there is no such work  on general polygonal/polyhedral meshes. We note that the classical analysis of HDG methods on simplex meshes \cite{BJRUH,BJFJPB} is hard to extend to polygonal  meshes; one can see \cite{BCWQKS} for more details.

In this paper, we shall develop a class of new EDG methods for the model problem \eqref{elliptic} on polygonal/polyhedral meshes.    Compared with the original EDG methods in \cite{BJSCHK},   our methods are of the following features.
\begin{itemize}
\item The new methods   use piecewise polynomials of degrees $k+1$, $k+1$, $k$ ($k\geq 0$) to approximate the potential, numerical trace and flux, respectively.

\item  Optimal error estimates are derived for both the potential and flux approximations.

  \item Our analysis   is based on polygonal/polyhedral meshes. The analysis technique here is due to \cite{LXAHDG}, where  a family of HDG methods for  \eqref{elliptic} on simplex meshes were analyzed  under the minimal regularity condition.

\end{itemize}

 The rest of this paper is organized as follows. In Section 2 we introduce   notation. Section 3 describes the   EDG scheme. Section 4 is devoted to the   error estimation  of  the proposed EDG methods. Finally, Section 5 provides some numerical results to verify the theoretical analysis.
%

\section{Notation}
For an arbitrary open set $D\subset\mathbb{R}^d$, we denoted by $H^1(D)$ the Sobolev space of scalar functions on $D$ whose derivatives up to order 1 are square integrable, with the norm  $\Vert\cdot\Vert_{1,D}$. The notation $\vert\cdot\vert_{1,D}$ denotes the semi-norm derived from the partial derivatives of order equal to 1. The space $H_0 ^1(D)$ denotes the closure in $H^1(D)$ of the set of infinitely differentiable functions with compact supports in $D$. We use $( \cdot , \cdot )_D$ and $\langle\cdot ,\cdot\rangle _{\partial D}$ to denote the ${L^2}$-inner products on the square integrable function spaces  ${L^2}(D)$ and ${L^2}(\partial D)$, respectively, with ${\left\|  \cdot  \right\|_D}$ and ${\left\|  \cdot  \right\|_{\partial D}}$ representing the corresponding induced ${L^2}$-norms. Let ${P_k}(D)$ denote the set of polynomials of degree $\leq k$ defined on $D$. 

 Let  $\mathcal{T}_h=\bigcup \{T\}$ be a conforming and shape regular subdivision of $\Omega$ into   convex polygons ($d$=2) or polyhedron ($d$=3), with $h_T$ being the diameter of $T$  and  $h:=\max\limits_{T\in \mathcal{T}_h}\{h_T\}$. 
Here `shape regular' is  in the sense that the following two assumptions \textbf{M1-M2} hold  \cite{Chen-Xie-arXiv2017}.

\begin{itemize}
	\item \textbf{M1} (Star-shaped elements). There exists a positive constant $\theta_*$ such that the following holds: for each element $T\in\mathcal{T}_h$, there exists a point $M_T\in T$ such that $T$ is star-shaped with respect to every point in the circle (or  sphere) of center $M_T$ and radius $\theta_* h_T$.
	
	\item \textbf{M2} (Edges or faces). There exists a positive constant $l_*$ such that: every element $T\in\mathcal{T}_h$, the distance between any two vertexes  is no less than $  l_* h_T$.
\end{itemize}

The regularity parameter of $\mathcal{T}_h$ is defined by $\rho:=\max\limits_{T\in \mathcal{T}_h}\{h_T ^d /|T| \}$, where $|T|$ is the $d$-dimension Lebesgue measure of $T$. Let $\mathcal{F}_h$ denote the set of all edges/faces of $\mathcal{T}_h$, and set $\partial \mathcal{T}_h:=\{ \partial T\ :\ T\in \mathcal{T}_h \}$. 

Based on the subdivision $\mathcal{T}_h$, we   introduce an auxiliary simplicial mesh $\mathcal{T}_h^*$ as follows:

\begin{itemize}
	\item When $d=2$, for any $T\in\mathcal{T}_h$,
	we connect $M_T$ and all $T$'s vertexes to divide $T$ into a set of triangles, denoted by $w(T)$.
	
	\item When $d=3$,  for any $T, T'\in \mathcal{T}_h$ and every face $F\subset \partial T\cap \partial T'$, we choose any vertex $A$ on $F$ and connect A to the rest of $F$'s vertexes to get a set of triangles, $v(F)$, and if $F\cap \partial \Omega \neq \varnothing$, we can get a set of triangles $v(F)$ by the same way. Finally we connect $M_T$ and every $v(F)$ to get a set of tetrahedrons, $w(T)$.
	
	\item We set $\mathcal{T}^*_h:=\bigcup_{T\in\mathcal{T}_h}w(T)$ for $d=2,3$. We note that $\mathcal{T}_h^*$ is shape regular due to \textbf{M1} and \textbf{M2}.
\end{itemize}

 For any $T\in \mathcal{T}_h,$ set
 $\partial T^*:=\bigcup_{T'\in w(T) }  \left\{\partial T' \cap \partial T\right\}, $ and define
 $$ \partial \mathcal{T}_{h^*}:=\{ \partial T^*: \ T\in \mathcal{T}_h\},$$ 
 $$ \mathcal{F}_h^*:=\left\{ F: F \text{ is an edge/face of }\mathcal{T}_h^*  \text{ and $F\subset F'$ for some $F'\in \mathcal{F}_h$} \right\}.$$ 
 Notice that when $d=2$ or $\mathcal{T}_h$ is a tetrahedron mesh for $d=3$, it holds 
 $$\partial T^*= \partial T,  \quad \partial \mathcal{T}_{h^*}=\partial\mathcal{T}_h, \quad \mathcal{F}_h^*=\mathcal{F}_h.$$  
And, when $d=3$ and $\mathcal{T}_h$ is a polyhedral mesh, $\partial T^*$ is the set of  triangles, into which each face $F\subset \partial T$ is subdivided.

We also need the broken Soblev space
$$H^s (\mathcal{T}_h) :=\{v\in L^2 (\Omega): v|_T \in H^s (T) ,\forall T \in \mathcal{T}_h\},$$
with  the norm $\|\cdot\|_{s,\mathcal{T}_h} $ defined by
$$\|v\|_{s,\mathcal{T}_h} ^2 :=\sum\limits_{T\in \mathcal{T}_h} \|v\|_{s,T} ^2,\quad \forall v\in H^s (\mathcal{T}_h) .$$
The broken Soblev space $H^s(\mathcal{T}_h^*)$ is defined similarly.


Throughout this paper, $x\lesssim y$($x\gtrsim y$) means
x$\leq Cy$($x\geq Cy$), where C denotes a positive constant that only depends on $d$, $k$, $\Omega$, the regularity parameter $\rho$, and the coefficient matrix $\bm c$. The notation $x\sim y$ abbreviates
$x\lesssim y\lesssim x$.

\section{ EDG method}

 For any $T\in \mathcal{T}_h$ and $F\in \mathcal{F}_h^*$, let  $V(T)\subset L^2(T)$, $\mathbf{W}(T)\subset [L^2(T)]^d$ and $M(F)\subset L^2(F)$ be local finite dimensional spaces. Then we define
\begin{eqnarray}
V_h  &:= & \{v_h\in L^2(\Omega):v_h|_T \in V(T),\forall T\in \mathcal{T}_h\},\\
\mathbf{W}_h &:=&\{\bm \tau_h \in [L^2(\Omega)]^d :\bm \tau_h|_T \in \mathbf{W}(T),\forall T\in \mathcal{T}_h\},\\
{M}_h &:=&\{\mu_h \in L^2(\mathcal{F}_h^*):\mu_h|_F\in M(F),\forall F\in \mathcal{F}_h^* \},\\
\widetilde{M}_h& :=&\{\mu_h \in C^0(\mathcal{F}_h^*):\mu_h|_F\in M(F),\forall F\in \mathcal{F}_h^* \},\\
\widetilde{M}_h(g) &:=&\{\mu_h \in \widetilde{M}_h:\mu_h|_{\partial \Omega}= \Pi_h^\partial g \},
\end{eqnarray}
where $\Pi_h^{\partial}$ is   a continuous interpolation operator from $ L^2(\partial \Omega)$ to $ C^0(\partial\Omega)\cap P_{k+1}(\mathcal{F}_h^*\cap \partial\Omega)$, which will be defined in the next section. 

Then the  variational formulations of the EDG method are given as follows:
Seek $(u_h,\widetilde {\lambda}_h,\bm \sigma_h)\in V_h\times \widetilde {M}_h (g) \times \mathbf{W}_h$ such that
\begin{eqnarray} \label{model1}
(\bm c \bm \sigma_h,\bm \tau_h)+(u_h,\text{div}_h \bm \tau_h)-\sum_{T\in\mathcal{T}_h}\langle\widetilde {\lambda}_h,\bm \tau_h\cdot n\rangle_{\partial T^*}=0&\forall \bm \tau_h\in \mathbf{W}_h,
\\
 \label{model2}
-(v_h,\text{div}_h\bm \sigma_h)+\sum_{T\in \mathcal{T}_h} \langle\alpha_T (u_h-\widetilde {\lambda}_h),v_h\rangle_{\partial T^*}=(f,v_h)&\forall v_h\in V_h,
\\
\label{model3}
\sum_{T\in \mathcal{T}_h} \langle \bm \sigma_h \cdot n -\alpha_T(u_h-\widetilde {\lambda}_h),\widetilde{\mu}_h \rangle_{\partial T^*}=0&\forall \widetilde {\mu}_h\in \widetilde {M}_h(0).
\end{eqnarray}
Here the broken operator $\text{div}_h$ is defined    by $\text{div}_h \bm \tau_h|_T :=\text{div}(\bm \tau_h|_T)$   for any $\bm \tau_h\in \mathbf{W}_h,T\in \mathcal{T}_h$.

In this paper we choose the local spaces $V(T)$, $M(F)$, $\mathbf{W} (T)$ and the penalty parameter $\alpha_T$ as following: for integer $k\geq 0$,
\begin{eqnarray}\label{local spaces}
V(T)\ &=&\ P_{k+1} (T),~~~~~~M(F)\ =\ P_{k+1} (F), ~~~~~\mathbf{W} (T)\ =\ [P_k (T)]^d,\\
\label{penalty parameter}
\alpha_T |_F\ &=&\ h_T ^{-1},~~\forall\  face\ F\ of\ T.
\end{eqnarray}

We have the following existence and uniqueness result:
\begin{lem}
	The EDG method (\ref{model1})-(\ref{model3}) admits a unique solution $(u_h,\widetilde\lambda_h,\bm \sigma_h)\in V_h\times \widetilde{M}_h (g) \times \mathbf{W}_h$.
\end{lem}
\begin{proof}
	It suffices  to prove the uniqueness, or equivalently,   to show that the system has the trivial solution when $f=g=0$.
	
	In fact, $f=g=0$ implies $\widetilde{\lambda}_h\in \widetilde{M}_h^0$. By taking $(\bm \tau_h,v_h,\widetilde{\mu}_h)=(\bm \sigma_h,u_h,\widetilde{\lambda}_h)$ in \eqref{model1}-\eqref{model3},  and summing all equations together, one can obtain
	\begin{align*}
	(\bm c\bm \sigma_h,\bm \sigma_h)+\sum_{T\in \mathcal{T}_h}\alpha_T|| u_h-\widetilde{\lambda}_h ||_{\partial T^*}^2=0.
	\end{align*}
		Since $\bm c$ is uniformly positive and $\alpha_T$ is nonnegative, the above equation implies $\bm \sigma_h=0$ and
	$ u_h=\widetilde{\lambda}_h$ on $\partial T$ for all $T\in \mathcal{T}_h$. Then, 	
  taking $\bm \tau_h=\nabla u_h$ in \eqref{model1} yields
	\begin{equation*}
	0=(u_h,\text{div}_h \nabla u_h)_{\mathcal{T}_h}-\sum_{T\in\mathcal{T}_h}\langle\widetilde {\lambda}_h,\nabla u_h\cdot n\rangle_{\partial T^*}=-(\nabla u_h, \nabla u_h).
	\end{equation*}
This means $\nabla u_h=0$ on each  $T\in \mathcal{T}_h$, i.e. $u_h$ is piecewise constant. 
	Recalling that  $u_h=\widetilde{\lambda}_h$  on $\mathcal{F}_h^*$ and $\widetilde{\lambda}_h=0$ on $\partial\Omega$, we  finally  obtain $u_h=0$  and $\widetilde{\lambda}_h=0$.
\end{proof}

Introduce the following two local problems:

For any $T\in \mathcal{T}_h$ and $\lambda_h \in L^2(\partial T^*)$, seek $(u_{\lambda_h} ,\bm \sigma_{\lambda_h})\in V(T)\times \mathbf{W} (T)$ such that
\begin{eqnarray} \label{local problem11}
(\bm c \bm \sigma_{\lambda_h} ,\bm \tau_h)_T +(u_{\lambda_h} ,\text{div} \bm \tau_h)_T =\langle\lambda_h ,\bm \tau_h \cdot n\rangle_{\partial T^*}&\forall \bm \tau_h \in \mathbf{W}(T),
\\
 \label{local problem12}
-(v_h,\text{div}\bm \sigma_{\lambda_h})_T +\langle\alpha_T u_{\lambda_h},v_h\rangle_{\partial T} = \langle\alpha_T \lambda_h,v_h \rangle_{\partial T^*}&\forall v_h\in V(T).
\end{eqnarray}

For any $T\in \mathcal{T}_h$ and $f \in L^2 (T)$, seek $(u_f ,\bm \sigma_f)\in V(T)\times \mathbf{W} (T)$ such that
\begin{eqnarray} \label{local problem21}
(\bm c \bm \sigma_f ,\bm \tau_h)_T +(u_f ,\text{div}\bm \tau_h)_T =0&\forall  \bm \tau_h \in \mathbf{W}(T),
\\
  \label{local problem22}
-(v_h,\text{div}\bm \sigma_f)_T +\langle\alpha_T u_f,v_h\rangle_{\partial T} = (f,v_h)_T&\forall v_h\in V(T).
\end{eqnarray}

Similar to the HDG method, after the local elimination of   unknowns $u_h$ and $\bm\sigma_h$,   the EDG method  leads to the following reduced system: seek $\widetilde{\lambda}_h \in \widetilde{M}_h (g)$ such that
\begin{equation}
a_h(\widetilde{\lambda}_h,\widetilde{\mu}_h)=(f,v_{\widetilde{\mu}_h})~~~\forall \widetilde{\mu}_h \in \widetilde {M}_h(0).
\end{equation}
where $a_h(\cdot,\cdot):\widetilde{M}_h(g)\times \widetilde{M}_h^0\longrightarrow
\mathbb{R}$ is defined by
\begin{equation} \label{bilinear_form}
a_h(\widetilde{\lambda}_h,\widetilde{\mu}_h):=\sum_{T\in \mathcal{T}_h} (\bm c \bm \sigma_{{\widetilde{\lambda}_h}},\bm \sigma_{\widetilde{\mu}_h})_T +\sum_{T\in \mathcal{T}_h} \langle\alpha_T (u_{{\widetilde{\lambda}_h}}-\widetilde{\lambda}_h),u_{{\widetilde{\mu}_h}}-\widetilde{\mu}_h\rangle_{\partial T^*}.
\end{equation} 

\begin{rem}
	Follow from \cite{BJRUH}, we can define an HDG method: Seek $(u_h,\lambda_h,\bm \sigma_h)\in V_h\times \widehat M_h (g) \times \mathbf{W}_h$ such that 
\begin{eqnarray*}
(\bm c \bm \sigma_h,\bm \tau_h)+(u_h,\text{div}_h \bm \tau_h)-\sum_{T\in\mathcal{T}_h}\langle\lambda_h,\bm \tau_h\cdot n\rangle_{\partial T}=0&\forall \bm \tau_h\in \mathbf{W}_h,
\\
-(v_h,\text{div}_h\bm \sigma_h)+\sum_{T\in \mathcal{T}_h} \langle\alpha_T (u_h-\lambda_h),v_h\rangle_{\partial T}=(f,v_h)&\forall v_h\in V_h,
\\
\sum_{T\in \mathcal{T}_h} \langle \bm \sigma_h \cdot n -\alpha_T(u_h-\lambda_h),\mu_h\rangle_{\partial T}=0&\forall \mu_h\in \widehat M_h (0).
\end{eqnarray*}
Here $$\widehat{M}_h(g) :=\{\mu \in L^2(\mathcal{F}_h):\mu|_F\in P_{k+1}(F),\forall F\in \mathcal{F}_h \text{ and } \ \langle \mu,\eta \rangle_{F}=\langle g,\eta \rangle _{F} \text{ if }F\subset \partial \Omega, \forall \eta \in P_{k+1}(F)  \}.$$
\end{rem}
\begin{rem}
	We can see that the EDG method is a modification of   the corresponding HDG method by simply replacing the discontinuous numerical trace space $\widehat M_h (g)$ with the  continuous trace space $\widetilde M_h (g)$. In particular, when $d=2$ or  $\mathcal{T}_h$ is a tetrahedron mesh,   $\widetilde M_h (g)$ is much smaller than $\widehat M_h (g)$. In such cases,  the EDG method  leads to a smaller system than the corresponding HDG method. 
%
\end{rem}

\section{Error analysis}
This section is devoted to the estimation of the flux  error $\bm  \sigma-\bm  \sigma_h$ and the potential error $u-u_h$ for the  EDG scheme \eqref{model1}-\eqref{model3}.  In subsections 4.1 and 4.2 we carry out the analysis for the flux and potential approximations, respectively on 2D/3D polygon meshes. 

\subsection{Estimation for flux approximation}

This subsection is devoted to the error estimation of the  flux approximation $\bm  \sigma_h$ for the EDG scheme \eqref{model1}-\eqref{model3}.

Let $P_V : L^2 (\Omega)\longrightarrow V_h$, $P_{\mathbf{W}}: [L^2 (\Omega) ]^d \longrightarrow \mathbf{W}_h$, and $P_M : L^2(\mathcal{F}_h^*) \longrightarrow M_h$ be the standard $L^2$-orthogonal projection operators.  Then the following estimates are standard. 
 \begin{lem} \label{lemma_L_2}
	For any $T\in \mathcal{T}_h$ and $(v, \bm \tau)\in H^{k+2} (T)\times [H^{k+1} (T)]^d$,  it holds
	\begin{equation} \label{projection_error}
	\begin{split}
	\| v-P_V v \|_T+h_T^{\frac{1}{2}}\|  v-P_V v \|_{\partial T}&\lesssim h_T^{k+2} | v |_{k+2,T},	\\
	\| \bm \tau-P_{\mathbf{W}} \bm \tau \|_T+h_T^{\frac{1}{2}}\| \bm \tau-P_{\mathbf{W}} \bm \tau \|_{\partial T}&\lesssim h_T^{k+1} | \bm \tau |_{k+1,T},\\
	\| v-P_M v \|_{\partial T^*} &\lesssim h_T^{k+\frac{3}{2}} | v |_{k+2,T}.
	\end{split}
	\end{equation}
\end{lem}

%

 For 
any $d$-simplex element $  T\in \mathcal{T}_h$ with vertices
$\bm a_j =(x_{1j},x_{2j},...x_{dj})^T \ (1\leq j \leq d+1),$
denote by
\begin{equation}
	S_{ T}:=\{ \bm x| \bm x=\sum_{i=1}^{d+1} \frac{j_i}{k+1} \bm a_i ,\sum_{i=1}^{d+1} j_i =k+1, j_i \in \{ 0,1,...,k+1\}, 1\leq i \leq d+1\}
\end{equation}
 the set of nodes of $T$ and by 
   $S_{\mathcal{T}_h}:=\bigcup\limits_{T\in \mathcal{T}_h}S_{ T} $ 
the set of nodes of $ \mathcal{T}_h$. Note that $S_T$ is   the set of nodes for the $C^0$  Lagrange finite element of order $k+1$. We let $\bm a_{i,T}$ be a node of Lagrange finite element of $S_T$ and $\bm a_{i,F}\in S_{\mathcal{T}_h}$, which lays on some edge/face $F\in \mathcal{F}_h$.

Given a point $\bm a\in \mathbb{R}^d$, we define 
\begin{align*}
	S_h ^c (\bm a) &:= \{T:\  T\in \mathcal{T}_h,\bm a\in T \}, \\
	S_h ^F (\bm a) &:= \{F:\  F\in \mathcal{F}_h^*\cap \partial\Omega,\bm a\text{ is on } F \}, 
\end{align*}
and let $\# S_h^c(\bm a)$ and $\#S_h^F(\bm a)$ be the number of elements in $S_h^c(\bm a)$ and $S_h^F(\bm a)$, respectively.

Now we define the continuous interpolation operator   $\Pi_h^\partial :L^2 (\partial\Omega) \longrightarrow C^0(\partial\Omega)\cap P_{k+1}(\mathcal{F}_h^* \cap \partial\Omega) $ as follows: For any $g\in L^2(\partial\Omega)$ and  $F\in \mathcal{F}_h^*\cap \partial \Omega$,   $\Pi_h^\partial g|_F\in M(F)$  satisfies
\begin{align*}
	\Pi_h^\partial g|_F(\bm a_{i, F})&= P_M g|_F (\bm a_{i,F}),\qquad \text{ if $\bm a_{i,F}$ is in the interior of }F,\\
	\Pi_h^\partial g|_F(\bm a_{i,F})&= \frac{1}{\# S_h^F(\bm a_{i,F})} \sum_{F'\in S_h^F(a_{i,F})} P_M g|_{F'} (\bm a_{i,F}),~ \text{ if }\bm a_{i,F} \text{ is a vertex of $F$}.
\end{align*}

 Following  Chapter 3 of \cite{SZCWM}, we introduce the projection mean operator   $\Pi_h ^P:L^2 (\Omega) \longrightarrow V_h \bigcap H^1 (\Omega) $, defined as follows: for any $T\in \mathcal{T}_h ,\ u\in L^2(\Omega),\ \Pi_h ^P u |_T \in V (T)$ and
\begin{align*}
\Pi_h ^P u |_T (\bm a_{i,T})&= (P_V u) |_T (\bm a_{i,T})\quad\qquad\qquad\qquad for\ any\ \bm a_{i,T} \text{ in the interior of $T$},\\
\Pi_h ^P u |_T (\bm a_{i,T})&=\frac{1}{\# S_h^c(\bm a_{i,T})} \sum_{T'\in S_h ^c (\bm a_{i,T})} {(P_V u)}|_{T'}(\bm a_{i,T'}) ~~  for\ any\ \bm a_{i,T}\ on\ \partial T, \partial T\cap \partial \Omega=\varnothing\\
\Pi_h ^P u |_T (\bm a_{i,T})&= \Pi_h^\partial g(\bm a_{i,T})\qquad\qquad\qquad\qquad\qquad\qquad\quad for\ any\ \bm a_{i,T}\ on\  \partial \Omega.
\end{align*}




When $\mathcal{T}_h$ is a polygonal/polyhedral subdivision, we can first define the  projection mean operator $\Pi_{h^*}^P$ on the auxilliary mesh $\mathcal{T}_{h}^*$ whose elements are simplexes. Then we define $\Pi_h^P$ as follows:
\begin{align*}
	\Pi_h^P u(\bm a)= \Pi_{h^*}^P u(\bm a)~~\forall \bm a\in S_{\mathcal{T}_{h^*}}\cap \mathcal{F}_h^*,~\forall u\in H^1(\Omega). 
\end{align*}


From  \cite{SZCWM} we have the following approximation result.
\begin{lem} \label{lemma_pihp}
 For any $u \in H^{k+2} (\mathcal{T}_h)$ and $T \in \mathcal{T}_h $, it holds
 \begin{equation}\label{meanproerro}
 ||u-\Pi_{h^*}^P u||_{\partial T^*} \lesssim h_T ^{k+\frac{3}{2}} \big(\sum_{T' \in \omega_T}|v|_{k+2,T'} ^2\big)^\frac{1}{2},
 \end{equation}
 where $ \omega_T:=\{ T'\in \mathcal{T}_h:~T'\cap T\neq \emptyset \}$.
\end{lem}


With the above projection operators, we set
\begin{align}
\delta^{\bm \sigma}&:= \bm \sigma-P_{\mathbf{W}} \bm \sigma,~~ \delta^u:=u- P_V u,~~~ \delta^{\widetilde{\lambda}}:=u- \Pi_h  u,\nonumber\\
e_h ^{\bm \sigma} &:= P_{\mathbf{W}} \bm \sigma -\bm \sigma_h,\ e_h ^u := P_V u -u_h,~\ e_h ^{\widetilde{\lambda}}:= \Pi_h u -\widetilde{\lambda}_h .\label{error operator}
\end{align}
For any given $(\bm \tau,v)\in [L^2(\Omega)]^d \times H^1 (\Omega)$, we define
\begin{equation} \label{L}
L_{\bm \tau,v} (\psi) :=\sum_{T\in \mathcal{T}_h} \langle(P_W \bm \tau -\bm \tau)\cdot \bm n -\alpha _T ( P_V v -\Pi _h^P v),\psi\rangle_{\partial T^*},  \ \forall \psi \in H^1 (\Omega) \bigcup \widetilde{M}_h \bigcup V_h .
\end{equation}
Then we have the following error equations.
\begin{lem} \label{errorequation}
For all $( \bm \tau_h, v_h,\widetilde{\mu}_h )\in \mathbf{W}_h\times V_h\times \widetilde {M}_h(0)$ it holds
\begin{equation}\label{errequation1}
(\bm c e_h ^{\bm \sigma} ,\bm \tau_h) +(e_h ^u,\text{div}_h \bm \tau_h)-\sum_{T\in \mathcal{T}_h} \langle e_h ^{\widetilde{\lambda}},\bm \tau_h \cdot n\rangle_{\partial T^*} =\ -(\bm c \delta^{\bm \sigma},\bm \tau_h)+\sum_{T\in \mathcal{T}_h} \langle P_M u-\Pi_h^P u,\bm \tau_h\cdot \bm n\rangle_{\partial T^*},
\end{equation}
\begin{equation}\label{errequation2}
-(\text{div}_h e_h ^{\bm \sigma},v_h)+\sum_{T\in \mathcal{T}_h} \langle\alpha_T ( e_h ^u -e_h ^{\widetilde{\lambda}}),v_h\rangle_{\partial T^*}=\ L_{\bm \sigma,u} (v_h),
\end{equation}
\begin{equation}\label{errequation3}
\sum_{T\in \mathcal{T}_h} \langle e_h ^{\bm \sigma} \cdot n -\alpha_T ( e_h ^u -e_h ^{\widetilde{\lambda}}),\widetilde{\mu}_h \rangle_{\partial T^*} =\ -L_{\bm \sigma,u} (\widetilde{\mu}_h).
\end{equation}
\end{lem}
\begin{proof}
In light of  \eqref{elliptic} and the definitions of $L^2$-orthogonal projection operators, we have,  for all $(\bm \tau_h ,v_h )\in \mathbf{W}_h \times V_h$, 
\begin{align*}
	(\bm c P_{\mathbf{W}} \bm \sigma ,\bm \tau_h)+(P_V u,\text{div}_h \bm \tau_h)-\sum_{T\in \mathcal{T}_h} \langle \Pi_h^P u,\bm \tau_h \cdot n\rangle_{\partial T^*}
	&=(\bm c(P_{\mathbf{W}}  \bm \sigma -\bm \sigma) , \bm \tau_h)\\
&\quad+\sum_{T\in \mathcal{T}_h} \langle P_M u-\Pi_h^P u,\bm \tau_h\cdot \bm n\rangle_{\partial T^*},\\
	(P_{\mathbf{W}}  \bm \sigma,\nabla v_h)_{\mathcal{T}_h}+\sum_{T\in \mathcal{T}_h} \langle P_{\mathbf{W}} \bm \sigma \cdot \bm n, v_h \rangle_{\partial T^*} &= (f,v_h)_{\mathcal{T}_h} +\sum_{T\in \mathcal{T}_h} \langle (P_{\mathbf{W}}\bm \sigma-\bm \sigma)\cdot \bm n,v_h\rangle_{\partial T}.
\end{align*}
By subtracting the above two equations from (\ref{model1}) and (\ref{model2}), respectively, we then obtain (\ref{errequation1}) and (\ref{errequation2}). Finally, equation (\ref{errequation3}) follows form (\ref{model3}) and the relation
\begin{align}
	\sum_{T\in \mathcal{T}_h} \langle \bm \sigma \cdot \bm n ,\widetilde {\mu}_h\rangle _{\partial T^*} =0,\ \forall \widetilde {\mu}_h \in \widetilde {M}_h(0).
\end{align}
\end{proof}

Introduce  a seminorm $|||\cdot|||:V_h\times \widetilde {M}_h(0) \times \mathbf{W}_h \longrightarrow \mathbb{R}$  with
\begin{equation}\label{|||-norm}
|||(v_h,\widetilde {\mu}_h,\bm \tau_h)|||^2:= (\bm c\bm \tau_h,\bm \tau_h) +\sum\limits_{T\in \mathcal{T}_h} ||\alpha_T ^\frac{1}{2} ( v_h - \widetilde{\mu}_h)||_{\partial T^*} ^2 , \  \forall (v_h,\widetilde {\mu}_h,\bm \tau_h) \in V_h\times \widetilde {M}_h(0) \times \mathbf{W}_h,
\end{equation}
then we easily get  the following lemma.

 \begin{lem}\label{errsig} It holds
 \begin{equation}\label{errsigma}
 \begin{split}
 |||(e_h ^u ,e_h ^{\widetilde{\lambda}},e_h ^{\bm \sigma})|||^2 \lesssim ||\bm \sigma-P_{\mathbf{W}} \bm \sigma||_{\mathcal{T}_h}^2+ \sum_{T\in \mathcal{T}_h}  h_T||\bm \sigma-P_{\mathbf{W}} \bm \sigma ||_{\partial T^*} ^2 + \sum_{T\in \mathcal{T}_h} h_T^{-1}|| \Pi_h^P u-P_M u ||_{\partial T^*} ^2.
 \end{split}
 \end{equation}
 \end{lem}
\begin{proof}
We first show 
\begin{equation}\label{errordec}
|||(e_h ^u ,e_h ^{\widetilde{\lambda}},e_h ^{\bm \sigma})|||^2=I_1 +I_2+I_3,
\end{equation}
where $ 
I_1 :=\  -(\bm c \delta^{\bm \sigma} ,e_h ^{\bm \sigma}),$ 
$I_2:=\  \sum_{T\in \mathcal{T}_h} \langle P_M u-\Pi_h^P u,\bm e_h^{\bm \sigma}\cdot \bm n\rangle_{\partial T^*},$
and $I_3:=\  L_{\sigma,u} (e_h^u -e_h^{\widetilde{\lambda}})$.

In fact, taking $\bm \tau_h=e_h ^{\bm \sigma}$ in (\ref{errequation1}), $v_h=e_h ^u$ in (\ref{errequation2}), $\widetilde{\mu}_h=e_h^{\widetilde{\lambda}}$ in (\ref{errequation3}), and adding the resultant three equations together, we obtain
\begin{align*}
	(& \bm c e_h ^{\bm \sigma} ,e_h^{\bm \sigma}) +(e_h ^u,\text{div}_h e_h^{\bm \sigma})-\sum_{T\in \mathcal{T}_h} \langle e_h ^{\widetilde{\lambda}},e_h^{\bm \sigma} \cdot n\rangle_{\partial T^*}-(\text{div}_h e_h ^{\bm \sigma},e_h^u)\\
	&+\sum_{T\in \mathcal{T}_h} \langle\alpha_T ( e_h ^u - e_h ^{\widetilde{\lambda}}),e_h^u\rangle_{\partial T^*}+\sum_{T\in \mathcal{T}_h} \langle e_h ^{\bm \sigma} \cdot n -\alpha_T ( e_h ^u -e_h ^{\widetilde{\lambda}}),e_h^{\widetilde{\lambda}} \rangle_{\partial T^*}\\
	&=|||(e_h ^u ,e_h ^{\widetilde{\lambda}},e_h ^{\bm \sigma})|||^2,
\end{align*}
which, together with Lemma \ref{errorequation}, yields  (\ref{errordec}).

In view of Cauchy-Schwarz inequality and the trace inequality,  it is easy to   get
\begin{align*}
I_1 &\lesssim ||\delta^{\bm \sigma}||_{\mathcal{T}_h} |||(e_h ^u ,e_h ^{\widetilde{\lambda}},e_h ^{\bm \sigma})|||, \\
I_2 &\lesssim (\sum_{T\in \mathcal{T}_h} h_T^{-1}|| (P_M u-\Pi_h^P u) ||_{\partial T^*} ^2)^{\frac{1}{2}} \ |||(e_h ^u ,e_h ^{\widetilde{\lambda}},e_h ^{\bm \sigma})|||,\\
I_3 &\lesssim (\sum_{T\in \mathcal{T}_h}h_T||\delta^{\bm \sigma}||_{\partial T^*}^2 +h_T^{-1}||P_V u-\Pi_h^P u||_{\partial T^*}^2)^\frac{1}{2}  |||(e_h ^u ,e_h ^{\widetilde{\lambda}},e_h ^{\bm \sigma})|||.
\end{align*}
Finally,  the desired estimate \eqref{errsigma} follows from    \eqref{errordec} and the above three inequalities.
\end{proof}

Based on the above lemmas, we easily derive the following  error estimate for the flux approximation.
 \begin{thm}\label{mainth}
	Let $(u, \bm \sigma)\in H^{k+2} (\mathcal{T}_h)\times [H^{k+1} (\mathcal{T}_h)]^d$ be the weak solution to the model \eqref{elliptic}  with $k\geq0$, and let $(u_h,\widetilde {\lambda}_h,\bm \sigma_h)\in V_h\times \widetilde {M}_h (g) \times \mathbf{W}_h$ be the solution to the EDG scheme (\ref{model1})-(\ref{model3}).  Then we have
	\begin{eqnarray}
	||\bm  \sigma -\bm  \sigma_h||&\lesssim& h ^{k+1} (||\bm \sigma||_{k+1 ,\mathcal{T}_h} + ||u||_{k+2 ,\mathcal{T}_h}). \label{flux}
	\end{eqnarray}
\end{thm}
\begin{proof}
The desired estimate \eqref{flux} follows from 
	the triangle inequality
	\begin{align*}
		||\bm \sigma - \bm \sigma_h|| \leqslant ||\bm \sigma - P_{\mathbf{W}} \bm \sigma||+||P_{\mathbf{W}} \bm \sigma - \bm \sigma_h||,
	\end{align*}
the defintion \eqref{|||-norm} of the seminorm $|||\cdot|||$, and  Lemmas \ref{lemma_L_2}, \ref{lemma_pihp} and \ref{errsig}. 
\end{proof}

 \subsection{Estimation for potential approximation}
 Based on the error estimation for the flux approximation $\bm \sigma_h$ in the previous subsection,   we  shall use the Aubin-Nitsche's technique of duality argument to derive the   estimation for the potential approximation $u_h$.  First, we introduce the following auxilliary problem:
 \begin{equation}\label{dual}
 \left\{\begin{array}{l}
  \bm c \bm \Phi -\nabla \Psi =0 ~~~~~\mbox{in}~~ \Omega ,\\
  \nabla\cdot \bm \Phi =  e_h ^u ~~~~~~~~\mbox{in}~~ \Omega ,\\
  \Psi = 0, ~~~~~~~~~~~~~on~~ \partial \Omega ,
  \end{array}\right.
\end{equation}
where, as defined in (\ref{error operator}), $e_h ^u =P_V u-u_h$. In addition, we assume the following regularity property holds:
\begin{equation} \label{regularity}
||\bm \Phi||_{1,\Omega} +||\Psi||_{2,\Omega} \lesssim ||e_h ^u||_{0,\Omega} .
\end{equation}
We have the following equality.
 \begin{lem}\label{erru decompose} It holds
 \begin{eqnarray}
 ||e_h ^u|| ^2 &=&(\bm c e_h^{\bm \sigma},\delta^{\bm \Phi})+(\bm c \delta^{\bm \sigma} ,P_{\mathbf{W}}\bm \Phi)+\langle e^u_h -e_h^{\widetilde{\lambda}},\delta^{\bm \Phi}\cdot\bm n\rangle_{\partial\mathcal T_{h^*}} 
	+\sum_{T\in \mathcal{T}_h}\langle\alpha_T(e^u_h-e^{\widetilde{\lambda}}_h),P_V\Psi-\Pi_h^P\Psi\rangle_{\partial T^*}\nonumber\\
	&&\quad - \langle  e_h^{\bm \sigma}\cdot\bm n,\delta^{\widetilde{\Psi}}\rangle_{\partial \mathcal{T}_{h^*}} + \langle P_M u-\Pi_h^P u,\delta^{\bm \Phi}\cdot \bm n\rangle_{\partial \mathcal{T}_{h^*}}-L_{\bm \sigma,u}(\Pi_h^P \Psi-P_V \Psi).
 \end{eqnarray}
 where
  $
	\delta^{\bm \Phi}:=\bm \Phi-P_{\mathbf{W}}\bm \Phi,$ $ \delta^{\Psi}:=\Psi-P_V \Psi,$ $ \delta^{\widetilde{\Psi}}:=\Psi-\Pi_h^P \Psi,
$
and $e_h^{\bm \sigma}, \delta^{\bm \sigma},e_h^{\widetilde{\lambda}}$ are defined in (\ref{error operator}).
 \end{lem}

 \begin{proof}
 By taking $\bm \tau_h=-P_{\mathbf{W}}\bm \Phi, v_h=P_V\Psi$, and $\widetilde{\mu}_h=\Pi_h^P \Psi$ in the error equations (\ref{errequation1})-(\ref{errequation3}), we can get
\begin{equation} \label{dual_error_1}
-(\bm c e_h ^{\bm \sigma} ,P_{\mathbf{W}}\bm \Phi) -(e_h ^u,\nabla\cdot P_{\mathbf{W}}\bm \Phi)+\sum_{T\in \mathcal{T}_h} \langle e_h ^{\widetilde{\lambda}},P_{\mathbf{W}}\bm \Phi \cdot n\rangle_{\partial T^*} = (\bm c \delta^{\bm \sigma},P_{\mathbf{W}}\bm \Phi)- \langle P_M u-\Pi_h^P u,P_{\mathbf{W}}\bm \Phi\cdot \bm n\rangle_{\partial \mathcal{T}_{h^*}},
\end{equation}
\begin{equation} \label{dual_error_2}
-(\nabla\cdot e_h ^{\bm \sigma},P_V\Psi)+\sum_{T\in \mathcal{T}_h} \langle\alpha_T (e_h ^u - e_h ^{\widetilde{\lambda}}),P_V\Psi\rangle_{\partial T^*}=\ L_{\bm \sigma,u} (P_V\Psi),
\end{equation}
\begin{equation} \label{dual_error_3}
\sum_{T\in \mathcal{T}_h} \langle e_h ^{\bm \sigma} \cdot n -\alpha_T (e_h ^u -e_h ^{\widetilde{\lambda}}),\Pi_h^P \Psi \rangle_{\partial T^*} =\ -L_{\bm \sigma,u} (\Pi_h^P \Psi).
\end{equation}
Integration by parts gives
\begin{align*}
-(e_h^u,\nabla\cdot P_{\mathbf{W}}\bm \Phi)=&-\langle e^u_h,P_{\mathbf{W}}\bm \Phi\cdot\bm n\rangle_{\partial\mathcal T_{h^*}}+(\nabla e^u_h,P_{\mathbf{W}}\bm \Phi)_{\mathcal T_h}\\
=&-\langle e^u_h,P_{\mathbf{W}}\bm \Phi\cdot\bm n\rangle_{\partial\mathcal T_{h^*}}+(\nabla e^u_h,\bm \Phi)_{\mathcal T_h}\\
=&\langle e^u_h,\delta^{\bm \Phi}\cdot\bm n\rangle_{\partial\mathcal T_{h^*}}-(e^u_h,\nabla\cdot\bm \Phi)_{\mathcal T_h}\\
=&\langle e^u_h,\delta^{\bm \Phi}\cdot\bm n\rangle_{\partial \mathcal T_{h^*}}-\|e^u_h\|^2_{\mathcal T_{h}}.
\end{align*}
Similarly, we can get
\begin{align*}
-(\nabla\cdot e_h ^{\bm \sigma},P_V\Psi)_{\mathcal{T}_h}&=(e_h ^{\bm \sigma},\nabla\Psi)_{\mathcal{T}_h}-\langle e_h^{\bm \sigma}\cdot\bm n, \Psi\rangle_{\partial \mathcal{T}_{h^*}}.
\end{align*}

Inserting the two equations above into \eqref{dual_error_1}-\eqref{dual_error_2}, we have
\begin{align}
-(\bm ce_h^{\bm \sigma},\bm\Phi)_{\mathcal T_h}&+\langle e^u_h,\delta^{\bm \Phi}\cdot\bm n\rangle_{\partial\mathcal T_{h^*}}-\|e^u_h\|^2_{\mathcal T_h}+\langle e^{\widetilde{\lambda}}_h,P_{\mathbf{W}}\bm \Phi\cdot\bm n\rangle_{\partial \mathcal{T}_{h^*}}\nonumber\\
&=(\bm c \delta^{\bm \sigma} ,P_{\mathbf{W}}\bm \Phi)_{\mathcal{T}_h}-(\bm c e_h^{\bm \sigma},\delta^{\bm \Phi})_{\mathcal T_h}+ \langle P_M u-\Pi_h^P u,P_{\mathbf{W}}\bm \Phi\cdot \bm n\rangle_{\partial \mathcal{T}_{h^*}},\label{dual_error1}\\
(e_h ^{\bm \sigma},\nabla\Psi)_{\mathcal{T}_h}&-\langle e_h^{\bm \sigma}\cdot\bm n, \Psi\rangle_{\partial \mathcal{T}_{h^*}}+\sum_{T\in \mathcal{T}_h} \langle\alpha_T (e_h ^u -e_h ^{\widetilde{\lambda}}),P_V\Psi\rangle_{\partial T^*}=\ L_{\bm \sigma,u} (P_V\Psi). \label{dual_error2}
\end{align}
Adding equations \eqref{dual_error_3}, \eqref{dual_error1}, and \eqref{dual_error2} together, and using the facts that $\bm \Phi-\nabla\Psi=0$ and $\langle e^{\widetilde{\lambda}}_h,\bm \Phi\cdot\bm n\rangle_{\partial\mathcal T_{h^*}}=0$, we obtain
\begin{align*}
&(e_h^{\bm \sigma},-\bm c \bm \Phi+\nabla\Psi)_{\mathcal T_h}+\langle e^u_h -e_h^{\widetilde{\lambda}},\delta^{\bm \Phi}\cdot\bm n\rangle_{\partial\mathcal T_{h^*}}-\|e^u_h\|^2_{\mathcal T_h}\\
&\quad+\sum_{T\in \partial\mathcal{T}_h}\langle\alpha_T(e^u_h-e^{\widetilde{\lambda}}_h),P_V\Psi-\Pi_h^P\Psi\rangle_{\partial T^*}- \langle  e_h^{\bm \sigma}\cdot\bm n, \delta^{\widetilde{\Psi}}\rangle_{\partial \mathcal{T}_{h^*}}\\
&=-(\bm c \delta^{\bm \sigma},P_{\mathbf{W}}\bm \Phi)_{\mathcal{T}_h}
 -(\bm c e_h^{\bm \sigma},\delta^{\bm \Phi})_{\mathcal T_h} +L_{\bm \sigma,u}(P_V \Psi-\Pi_h^P \Psi)\\
 &\quad- \langle P_M u-\Pi_h^P u,\delta^{\bm \Phi}\cdot \bm n\rangle_{\partial \mathcal{T}_{h^*}},
\end{align*}
which yields the desired conclusion.
\end{proof}

In light of Lemma \ref{erru decompose}, we further have the following estimate.  
 \begin{lem} \label{esti4I}
Under the regularity assumption (\ref{regularity}), it holds
 \begin{equation}
 	\begin{split}
 	||e_h^u||_{\mathcal{T}_h} \lesssim h|||(e_h ^u ,e_h ^{\widetilde{\lambda}},e_h ^{\bm \sigma})||| +h||\delta^{\bm \sigma}||_{\mathcal{T}_h} + h^{\frac{3}{2}} ||\delta^{\bm \sigma}||_{\partial \mathcal{T}_{h^*}}+h^\frac{1}{2}||P_V u-\Pi_h^P u||_{\partial \mathcal{T}_{h^*}}.
 	\end{split}
 \end{equation}
 \end{lem}
\begin{proof}  Set
$
 ||e_h ^u|| ^2 =:\Pi_1+\Pi_2+\Pi_3+\Pi_4+\Pi_5+\Pi_6+\Pi_7
$ 
 with 
 \begin{align*}
 \Pi_1 &:= (\bm c e_h^{\bm \sigma},\delta^{\bm \Phi})_{\mathcal T_h},\quad
 \Pi_2: = (\bm c \delta^{\bm \sigma} ,P_{\mathbf{W}}\bm \Phi),\quad
 \Pi_3: = \langle e^u_h -e_h^{\widetilde{\lambda}},\delta^{\bm \Phi}\cdot\bm n\rangle_{\partial\mathcal T_{h^*}},\\
 \Pi_4 &:= \sum_{T\in \mathcal{T}_h}\langle\alpha_T(e^u_h-e^{\widetilde{\lambda}}_h),P_V\Psi-\Pi_h^P\Psi\rangle_{\partial T^*},\quad
 \Pi_5: = -\langle  e_h^{\bm \sigma}\cdot\bm n, \delta^{\widetilde{\Psi}}\rangle_{\partial \mathcal{T}_{h^*}},\\
 \Pi_6 &:=  \langle P_M u-\Pi_h^P u,\delta^{\bm \Phi}\cdot \bm n\rangle_{\partial \mathcal{T}_{h^*}},\quad
 \Pi_7 := -L_{\bm \sigma,u}(\Pi_h^P \Psi-P_V \Psi).
  \end{align*}

In view of  Lemmas \ref{lemma_L_2}-\ref{lemma_pihp},  the assumption (\ref{regularity}),  and  Cauchy-Schwartz inequality, we obtain
	\begin{align*}
		|\Pi_1| &\lesssim \| e_h^{\bm \sigma} \|_{\mathcal{T}_h}\| \delta^{\bm \Phi} \|_{\mathcal{T}_h}\lesssim h\|e_h^{\bm \sigma}\|_{\mathcal{T}_h}\| \bm \Phi \|_{1,\Omega}\lesssim h|||(e_h^u,e_h^{\widetilde{\lambda}},e_h^{\bm \sigma})||| \| e_h^u \|_{\mathcal{T}_h},\\
		|\Pi_3| &\le \| e_h^u-e_h^{\widetilde{\lambda}} \|_{\partial \mathcal{T}_{h^*}}\| \delta^{\bm \Phi} \|_{\partial \mathcal{T}_h}\lesssim h^\frac{1}{2} \| e_h^u-e_h^{\widetilde{\lambda}} \|_{\partial \mathcal{T}_{h^*}}\| \bm \Phi \|_{1,\Omega}\lesssim h|||(e_h^u,e_h^{\widetilde{\lambda}},e_h^{\bm \sigma})||| \| e_h^u \|_{\mathcal{T}_h},\\
		|\Pi_4| &\le h^{-1}\| e_h^u-e_h^{\widetilde{\lambda}} \|_{\partial \mathcal{T}_{h^*}}\| P_V \Psi -\Pi_h^P \Psi \|_{\partial \mathcal{T}_{h^*}}\lesssim h|||(e_h^u,e_h^{\widetilde{\lambda}},e_h^{\bm \sigma})||| \| e_h^u \|_{\mathcal{T}_h},\\
		|\Pi_5| &\le \| e_h^{\bm \sigma} \|_{\partial \mathcal{T}_{h^*}}\| \delta^{\widetilde{\Psi}} \|_{\partial \mathcal{T}_{h^*}}\lesssim h\| e_h^{\bm \sigma} \|_{\mathcal{T}_h}\| \Psi \|_{2,\Omega}\lesssim h|||(e_h^u,e_h^{\widetilde{\lambda}},e_h^{\bm \sigma})||| \| e_h^u \|_{\mathcal{T}_h},\\
	|\Pi_6| &\lesssim  \| P_M u-\Pi_h^P u \|_{\partial \mathcal{T}_{h^*}} \|\delta^{\bm \Phi} \|_{\partial \mathcal{T}_{h^*}} \lesssim h^\frac{1}{2} \| P_M u-\Pi_h^P u \|_{\partial \mathcal{T}_{h^*}} \| e_h^u \|_{\mathcal{T}_h} \\		
		|\Pi_7| &\lesssim (\| \delta^{\bm \sigma} \|_{\partial \mathcal{T}_{h^*}}+h^{-1}\| P_V u- \Pi_h^P u \|_{\partial \mathcal{T}_{h^*}}) (\|P_V\Psi -\Pi_h^P \Psi  \|_{\partial \mathcal{T}_{h^*}})\\
		&\lesssim h^{\frac{3}{2}} \| \delta^{\bm \sigma} \|_{\partial \mathcal{T}_{h^*}}\| \Psi \|_{2,\Omega}+h^\frac{1}{2}\| P_V u- \Pi_h^P u \|_{\partial \mathcal{T}_{h^*}} \| \Psi \|_{2,\Omega}\\
		&\lesssim (h^\frac{3}{2}\| \delta^{\bm \sigma} \|_{\partial \mathcal{T}_{h^*}}+h^\frac{1}{2}\| P_V u-\Pi_h^P u \|_{\partial \mathcal{T}_{h^*}})\| e_h^u \|_{\mathcal{T}_h}
	\end{align*}
The thing left is to estimate $\Pi_2$, and we have
\begin{align*}
\Pi_2 &= (P_\mathbf{W} \bm \sigma -\bm \sigma ,\bm c P_\mathbf{W} \bm \Phi)\\
	 &=(P_\mathbf{W} \bm \sigma -\bm \sigma,\bm c (P_\mathbf{W} \bm \Phi-\bm \Phi)) +(P_\mathbf{W} \bm \sigma -\bm \sigma ,\bm c \bm \Phi)\\
	&=(P_\mathbf{W} \bm \sigma -\bm \sigma,\bm c (P_\mathbf{W} \bm \Phi-\bm \Phi))+(P_\mathbf{W} \bm \sigma -\bm \sigma, \nabla \phi-\nabla P_V \phi)\\
	&\lesssim h\| \delta^{\bm \sigma} \|_{\mathcal{T}_h}\| \bm\Phi \|_{1,\Omega}+h\| \delta^{\bm \sigma} \|_{\mathcal{T}_h} \| \Psi \|_{2,\Omega}\\
	&\lesssim h\| \delta^{\bm \sigma} \|_{\mathcal{T}_h}\| e_h^u \|_{\mathcal{T}_h}.
\end{align*}
Finally, combining   all the estimates of $\Pi_j$ ($j=1,\cdots,7)$ indicates the conclusion.
\end{proof}
 \begin{thm}\label{mainth}
	Let $(u, \bm \sigma)\in H^{k+2} (\mathcal{T}_h)\times [H^{k+1} (\mathcal{T}_h)]^d$ be the weak solution to model \eqref{elliptic}  with $k\geq0$, and let $(u_h,\widetilde {\lambda}_h,\bm \sigma_h)\in V_h\times \widetilde {M}_h (g) \times \mathbf{W}_h$ be the solution to the EDG scheme (\ref{model1})-(\ref{model3}).  Then, under the regularity assumption \eqref{regularity}, it holds
	\begin{eqnarray}
	||u-u_h||&\lesssim& h^{k+2} (||\bm \sigma||_{k+1,\mathcal{T}_h}+||u||_{k+2,\mathcal{T}_h}).\label{scalar}
	\end{eqnarray}
\end{thm}
\begin{proof} 
	From the  triangle inequality and  Lemmas \ref{esti4I}, \ref{errsig} and  \ref{lemma_L_2}, we have
	\begin{align*}
		\| u-u_h \|_{\mathcal{T}_h}\le \| e_h^u \|_{\mathcal{T}_h}+\| \delta^u \|_{\mathcal{T}_h} \lesssim h^{k+2} (| \bm \sigma |_{k+1}+| u |_{k+2}).
	\end{align*}
\end{proof}

\section{Numerical results}
In this section, we use a   two-dimensional numerical example to verify the   theoretical results.  We take  $\Omega = [0,1]\times [0,1]$, and let the exact solution to \eqref{elliptic}  be $u(x,y)=sin(\pi x)sin(\pi y)$ with the coefficient matrix 
\begin{equation} \label{define_c}
\bm c=
\left[
\begin{array}{ccc}
1+{x^2}{y^2} & 0\\
0 & 1+{x^2}{y^2}\\
\end{array}
\right].
\end{equation}

We consider two types of meshes:  uniform triangular meshes   and  quadrilateral meshes (Figure 1). Numerical results of the flux and potential approximations are listed in  Tables \ref{tab:ex1} and  \ref{tab:ex2} for the proposed EDG methods and the corresponding HDG methods with $k=0, 1, 2$.  We can see that both the HDG and EDG methods converge with the optimal rates. 

Table \ref{tab:ex3} shows the numbers of unknowns of the reduced system \eqref{bilinear_form} with $k=0, 1, 2$  which  contains the degrees of freedom of the numerical traces on interelement boundary as the only unknowns.  In this example, the EDG method always leads to  smaller systems than the corresponding HDG method.

 


\begin{table}[H]
	\begin{threeparttable}
		\caption{Convergence history on triangular meshes}\label{tab:ex1}	
		\begin{tabular}{|cc|cccc|cccc|}
			\hline
			&&&EDG&&&&HDG&&\\
			\hline
			$k$   &  Mesh          &$||u-u_h||$       &rate   &  $||\bm \sigma -\bm \sigma_h||$ &rate  &$||u-u_h||$       &rate   &  $||\bm \sigma -\bm \sigma_h||$ &rate   \\
			\hline
			0        &  $4\times4$       &8.076e-2   & -        &3.820e-1       &-            &1.940e-1 &-          &2.775e-1 &- \\
			&  $8\times8$       &2.084e-2   & 1.954   & 1.966e-1    &0.958    &4.917e-2 &1.980   &1.412e-1 &0.975 \\
			&  $16\times16$    &5.274e-3   & 1.982  & 9.907e-2     &0.989   &1.233e-2  &1.996  &7.089e-2 &0.994 \\
			&  $32\times32$   &1.323e-3   & 1.995   &4.963e-2     &0.997   &3.086e-3 &1.998   &3.549e-2 &0.998 \\
			&  $64\times64$   &3.310e-4   & 1.999  &2.483e-2     &0.999     &7.717e-4 &2.000   &1.775e-2 &1.000 \\
			\hline
			1       &  $4\times4$       & 2.305e-2  &-          & 4.770e-2        &-             &2.410e-2 &-        &4.232e-2 &\\
			&  $8\times8$      &  2.883e-3  &2.999    & 1.312e-2      &1.862      &3.031e-3  &2.991 &1.085e-2 &1.964\\
			&  $16\times16$   &  3.522e-4  &3.033      & 3.580e-3   &1.874     &3.791e-4   &2.999 &2.730e-3&1.991\\
			&  $32\times32$  &  4.336e-5  &3.022       & 9.358e-4  &1.936     &4.739e-5   &2.999 &6.839e-4&1.997\\
			&  $64\times64$  &  5.387e-6  &3.009      &2.377e-4     &1.977      &5.923e-6 &3.000  &1.711e-4  &1.999\\
			\hline
			2       &  $4\times4$       & 2.677e-3  &-          & 6.482e-3        &-             &2.880e-3 &-        &5.378e-3 &\\
			&  $8\times8$      &  1.746e-4 &3.938    & 8.074e-4      &3.005      &1.827e-4  &3.978 &6.874e-4 &2.968\\
			&  $16\times16$   &  1.106e-5  &3.981      & 1.004e-4   &3.008     &1.146e-5   &3.995 &8.650e-5&2.990\\
			&  $32\times32$  &  6.938e-7  &3.994       & 1.251e-5  &3.004     &7.166e-7   &3.999 &1.084e-5&2.997\\
			&  $64\times64$  &  4.342e-08  &3.998      &1.561e-6     &3.002      &4.479e-8 &4.000  &1.355e-6 &2.999\\
			\hline
		\end{tabular}
	\end{threeparttable}
\end{table}

\begin{figure}[H]
\centering
\includegraphics[width=0.49\textwidth,clip]{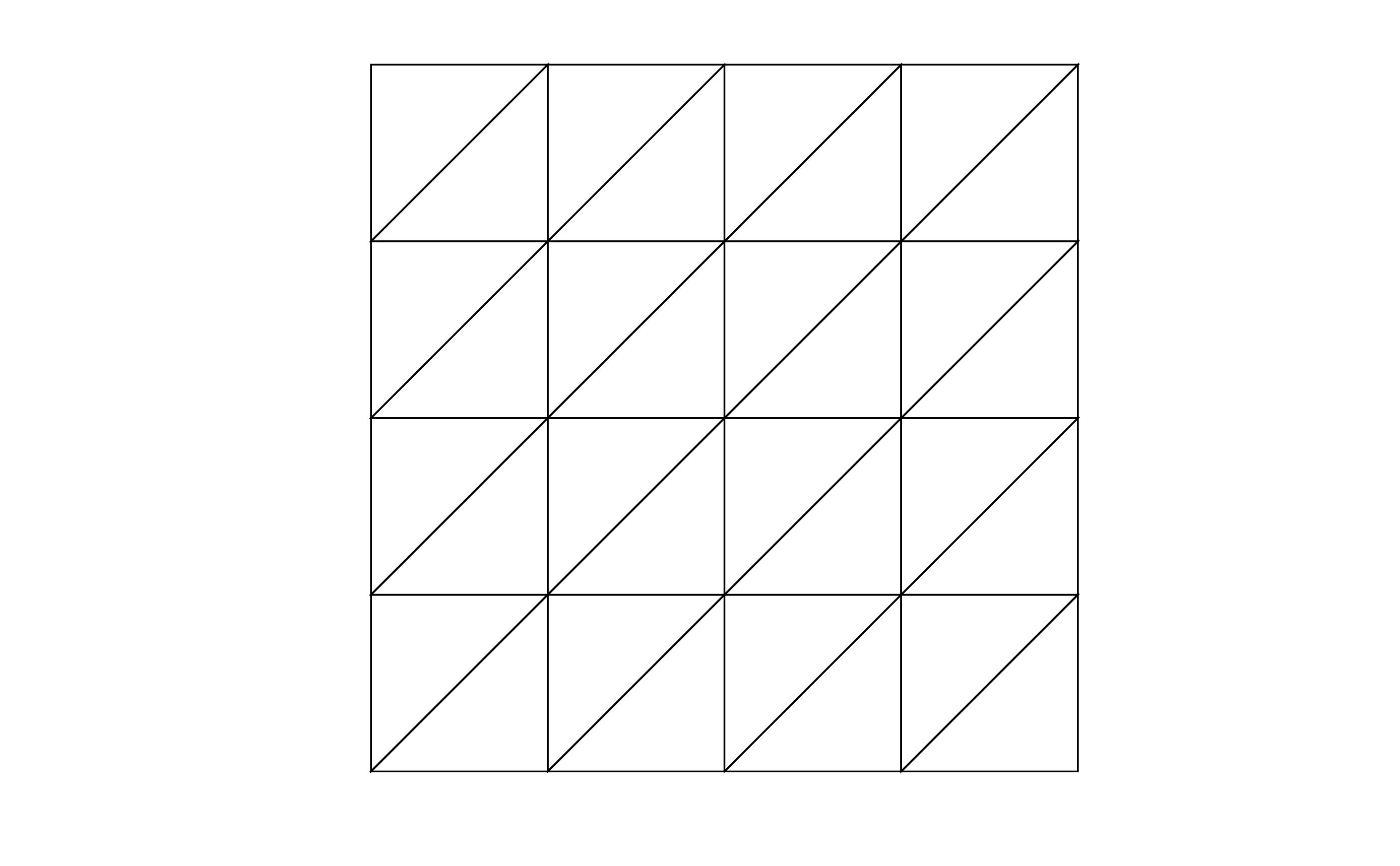}
\hfill
\includegraphics[width=0.49\textwidth,clip]{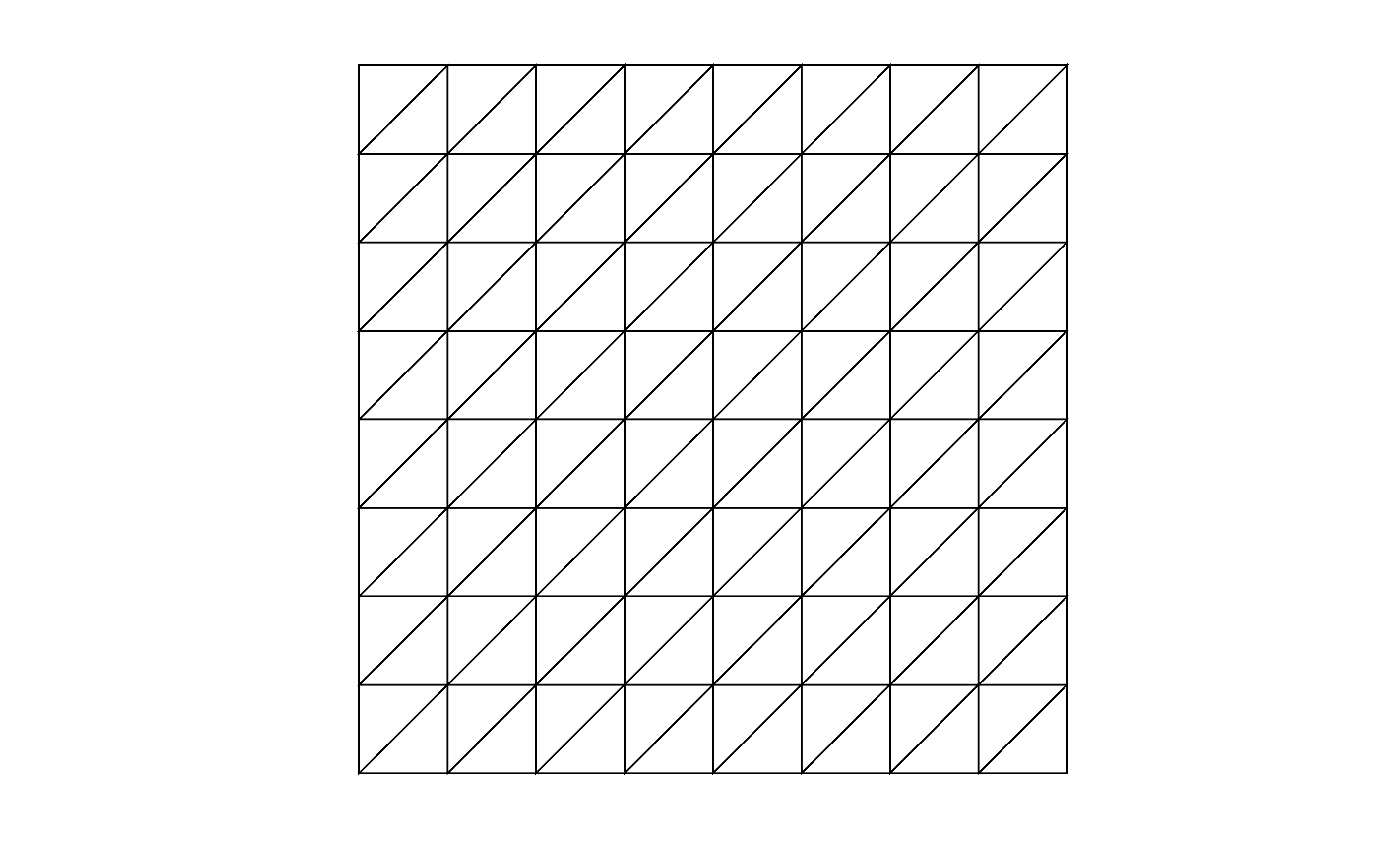}
\hfill
\includegraphics[width=0.49\textwidth,clip]{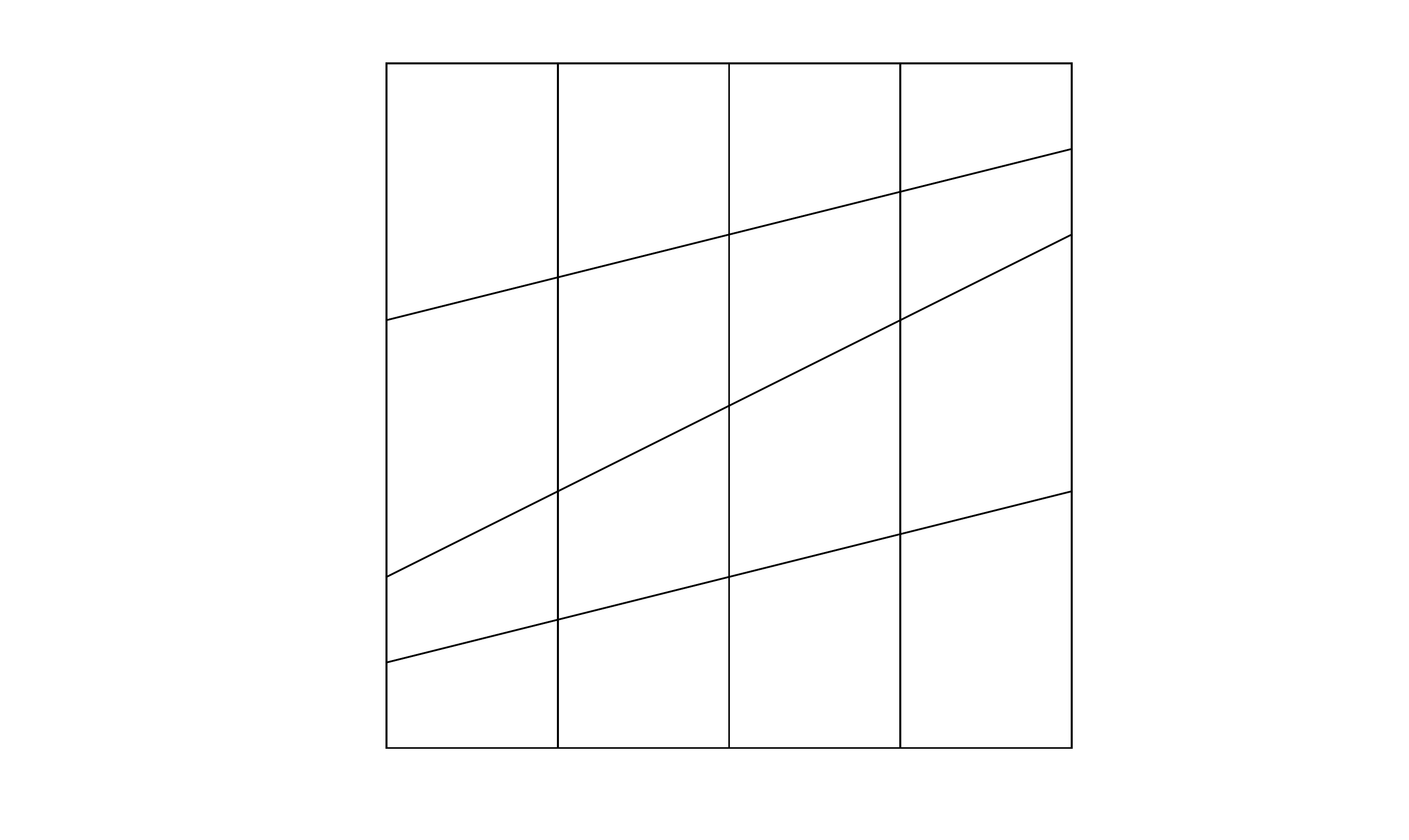}
\hfill
\includegraphics[width=0.49\textwidth,clip]{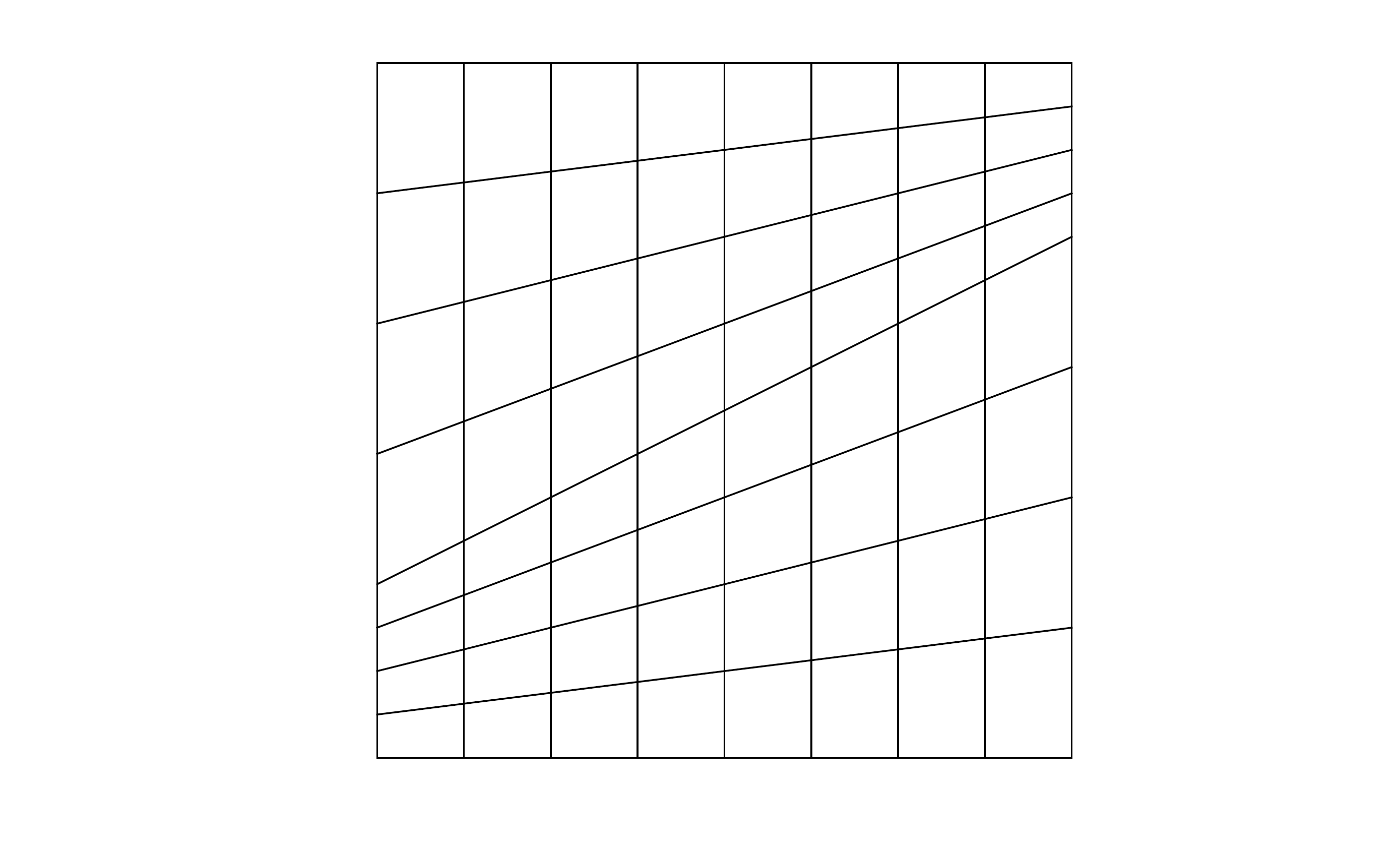}
\vspace{-0.1cm}
 \caption{\label{fig:nodes}{Two types of meshes: \emph{Left}: $4\times 4$; \emph{Right}: $8\times 8$.}}
\end{figure}

\begin{table}[H]
	\centering
	\begin{threeparttable}
		\caption{Convergence history on quadrilateral meshes }	\label{tab:ex2}
		\begin{tabular}{|cc|cccc|cccc|}
			\hline
			&&&EDG&&&&HDG&&\\
			\hline
			$k$   &  Mesh          &$||u-u_h||$       &rate   &  $||\bm \sigma -\bm \sigma_h||$ &rate    &$||u-u_h||$       &rate &  $||\bm \sigma -\bm \sigma_h||$ &rate\\
			\hline
			0        &  $4\times4$       &2.623e-1   & -        &3.324e-1       &-    &2.722e-1&- &3.362e-1&-      \\
			&  $8\times8$       & 6.917e-2  & 1.923   &1.706e-1    &0.962 &7.228e-2&1.913 &1.721e-1&0.950  \\
			&  $16\times16$    &1.752e-2   & 1.981  & 8.583e-2     &0.991  &1.845e-2&1.970 &8.621e-2&0.997 \\
			&  $32\times32$   &4.395e-3   & 1.995   &4.298e-2     &0.998  &4.644e-3&1.990 &4.305e-2&1.002 \\
			&  $64\times64$   &1.100e-3   & 1.998  &2.150e-2     &0.999    &1.163e-3&1.998 &2.151e-2&1.001\\
			\hline
			1       &  $4\times4$       &4.666e-2  &-          & 7.782e-2        &-    &4.498e-2&-  &8.057e-2&-       \\
			&  $8\times8$      &  6.023e-3  &2.960    & 1.973e-2      &1.982  &5.408e-3&3.056 &2.058e-2&1.919   \\
			&  $16\times16$   &  7.633e-4  &2.980      & 4.958e-3   &1.993   &6.592e-4&3.036 &5.094e-3&2.014  \\
			&  $32\times32$  &  9.590e-5  &2.993       & 1.241e-3  &1.998   &8.201e-5&3.007 &1.258e-3&2.018 \\
			&  $64\times64$  &  1.201e-5  &2.997      &3.104e-4     &1.999    &1.026e-5&3.000 &3.123e-4&2.010 \\
			\hline
			2       &  $4\times4$       &7.744e-3  &-          & 1.200e-2        &-    &7.720e-3&-  &1.219e-2&-       \\
			&  $8\times8$      &  4.942e-4  &3.970    & 1.387e-3      &3.113  &4.936e-4&3.967 &1.412e-3&3.111   \\
			&  $16\times16$   &  3.096e-5  &3.997      & 1.729e-4   &3.004   &3.101e-5&3.993 &1.749e-4&3.013 \\
			&  $32\times32$  &  1.933e-6  &4.002       & 2.139e-5  &3.015   &1.940e-6&3.999 &2.152e-5&3.023 \\
			&  $64\times64$  &  1.207e-7  &4.001     &2.659e-6    &3.008    &1.213e-7&4.000 &2.667e-6&3.012 \\
			\hline
		\end{tabular}
	\end{threeparttable}
\end{table}

\begin{table}[H]
	\centering
	\begin{threeparttable}
		\caption{Comparison of numbers of degrees of freedom }	\label{tab:ex3}
		\begin{tabular}{|cc|cc|cc|}
			\hline
					&	 &simplex &meshes &quadrilateral &meshes \\
			\hline
			$k$     &Mesh    &EDG&HDG  &EDG&HDG      \\
			\hline
			0        &  $4\times4$       &9	  &80        &9       &48          \\
			&  $8\times8$       & 49  & 352   &49    &224  \\
			&  $16\times16$    &225   & 1472 &225    &960  \\
			&  $32\times32$   &961   & 6016   &961     &3968   \\
			&  $64\times64$   &3936   & 24320  &3936     &16128    \\
			\hline
			1       &  $4\times4$       &49  &120          & 33  &72          \\
			&  $8\times8$      &  225  &528   & 161      &336    \\
			&  $16\times16$   &  961  &2208      & 705   &1440     \\
			&  $32\times32$  &  3936  &9024       & 2945  &5952    \\
			&  $64\times64$  &  16129  &36480      &12033     &24192   \\
			\hline
			2       &  $4\times4$       &89  &160          & 57  &96          \\
			&  $8\times8$      &  401  &704   & 273      &448    \\
			&  $16\times16$   &  1697  &2944      & 1185   &1920     \\
			&  $32\times32$  &  6977  &12032       & 4929  &7936    \\
			&  $64\times64$  &  28289  &48640      &20097     &32256   \\
			\hline
		\end{tabular}
	\end{threeparttable}
\end{table}

\end{document}